\documentclass[12pt]{article}
\usepackage[margin=25mm]{geometry}
\usepackage{amsmath}
\usepackage{amsfonts}
\usepackage{amssymb}
\usepackage{graphicx}
\usepackage{verbatim}
\immediate\write18{texcount -tex -sum  \jobname.tex > \jobname.wordcount.tex}

\usepackage{authblk}
\usepackage{amsthm}
\theoremstyle{plain}

\newtheorem{theorem}{Theorem}

\newtheorem*{corollary}{Corollary}

\theoremstyle{remark}
\newtheorem*{remark}{Remark}

\theoremstyle{definition}
\newtheorem{definition}{Definition}

\title{Vertex energies of splitting and shadow graphs}
\author[1]{Harshitha A}
\author[2]{Madhumitha K V}
\author[2]{Sabitha D'Souza}
\affil[1]{Department of Mathematics, Manipal Institute of Technology Bengaluru, Manipal Academy of Higher Education, Manipal, India}
\affil[2]{Department of Mathematics, Manipal Institute of Technology, Manipal Academy of Higher Education, Manipal, India, 576104}
\date{}

\begin{document}
	
	\maketitle
	\begin{abstract}
	Graph energy has been widely investigated in spectral graph theory. However, the manner in which this energy is distributed among individual vertices has received attention only more recently, following the introduction of energy of a vertex by O. Arizmendi et al. in 2018. Since then, several studies have explored this concept. In the present work, we focus on splitting graphs and shadow graphs and study how the total energy of a graph is distributed among its vertices. We further establish explicit expressions for vertex energies in these derived structures and analyze their behavior in comparison with the underlying graph.  
	\end{abstract}
	
\section{Introduction}	
The energy of a graph is introduced by Iven Gutman in 1978, as the sum of the absolute values of the eigenvalues of the adjacency matrix associated with the graph \cite{Ref1}. This concept attracted many researchers and has been studied extensively due to its novelty and applications in many fields. But recently the concept of energy of a vertex or vertex energy is introduced by O. Arizmendi et al., which gives the distribution of energy over all the vertices of the graphs \cite{2}. For a simple graph $G$ on $n$ vertices with $\lambda_1, \lambda_2,\ldots,\lambda_n$ as the eigenvalues of the adjacency matrix $A(G)$ of $G$, \begin{align*}
	\mathcal{E}(G)&=\sum\limits_{i=1}^{n}\left|\lambda_i\right|
\end{align*} 
is the energy of $G$.
Suppose $A(G)$ has the spectral decomposition
\[
A(G)=\sum_{i=1}^{n} \lambda_i\, u_i u_i^{T},
\]
where 
$u_i=(u_{i1},u_{i2},\dots,u_{in})^T$ are corresponding orthonormal eigenvectors.

The \emph{vertex energy} of a vertex $v_k \in V(G)$ is defined as
\[
\mathcal{E}_G(v_k)
=\sum\limits_{i=1}^{n} |\lambda_i|\, u_{ik}^2.
\]

Note that
\[
\mathcal{E}(G)=\sum_{k=1}^{n} \mathcal{E}_G(v_k).
\]

%  Quote references by using \cite. For example
%Arizmendi et al. \cite{3} established a fundamental theoretical basis by deriving the Coulson integral formula for the vertex energy of a graph. Further developments were made by Arizmendi and Sigarreta \cite{4}, who examined how the vertex energy changes under certain graph operations, specifically when joining trees. Ramane et al. \cite{5} investigated the vertex energy of subdivision graphs and derived explicit expressions for the vertex energies of such graphs. Yang et al. \cite{6} proposed an alternative technique for calculating vertex energy using the Estrada–Benzi approach, introducing efficient computational methods particularly suitable for large sparse graphs. Recently, Shrikanth et al. \cite{7} studied the vertex energy of small integral graphs, computing explicit values and classifying graphs based on vertex energy characteristics.

The concept of vertex energy has attracted considerable attention in recent years due to its ability to measure the spectral contribution of individual vertices in a graph. Arizmendi et al.~\cite{3} developed a Coulson integral formula for the vertex energy of a graph, analogous to the classical Coulson formula for graph energy, providing a theoretical foundation for the analytical evaluation of vertex energy. Arizmendi and Sigarreta~\cite{4} investigated the variation in vertex energy when joining trees and established several results describing how vertex energy behaves under tree-joining operations. Ramane et al.~\cite{5} examined the vertex energy of subdivision graphs and derived explicit expressions, highlighting the effect of subdivision operations on vertex-based spectral properties. Yang et al.~\cite{6} introduced a computational method based on the Estrada--Benzi approach, improving the calculation efficiency of vertex energy, particularly for large sparse graphs. Shrikanth et al.~\cite{7} studied the vertex energy of small integral graphs and classified such graphs in terms of their vertex energy characteristics, thereby revealing structural properties connected to integral spectra.

In this manuscript, we focus on the $m$-splitting graph and $m$-shadow graph associated with a given base graph $G$. We derive explicit expressions for the vertex energies of these derived graphs and express them in terms of the vertex energies of the original graph $G$. 

Further, in \cite{8}, the authors proved that the energy of the splitting graph is equal to $\sqrt{5}$ times the energy of the graph, and the energy of the shadow graph is equal to $2$ times the energy of the graph. In \cite{9}, the authors established that the energy of the $m$-splitting graph is equal to $\sqrt{4m+1}$ times the energy of the graph, while the energy of the $m$-shadow graph is equal to $m$ times the energy of the graph. In the present work, we verify and extend these results at the level of vertex energies.

\begin{definition}
	Let $G=(V,E)$ be a simple graph with the vertex set $V(G)=\{v_1,v_2,\dots,v_n\}$. 
	The splitting graph of $G$, denoted by $S'(G)$, is the graph obtained by adding for each vertex $v_i \in V(G)$
	a new vertex $v_i'$ (called the split vertex) and joining $v_i'$ to all the neighbors of $v_i$ in $G$.
	\[
	V(S'(G)) = V(G) \cup \{v_1', v_2', \dots, v_n'\},
	\]
	\[
	E(S'(G)) = E(G) \cup \{ v_i'v_j : v_i v_j \in E(G) \}.
	\]
	Thus, each split vertex $v_i'$ has the same open neighborhood as $v_i$ in $G$.
\end{definition}
\begin{definition}
	Let $G=(V,E)$ be a simple graph. The $m$-splitting graph of $G$, denoted by $Spl_m(G)$, 
	is the graph obtained by adding $m$ duplicate vertices for each vertex $v_i \in V(G)$, denoted by 
	$v_i^1, v_i^2, \dots, v_i^m$, and joining each $v_i^k$ to all the neighbors of $v_i$ in $G$.
	\[
	V(Spl_m(G)) = V(G) \cup \{ v_i^k : 1 \le i \le n,\, 1 \le k \le m \},
	\]
	\[
	E(Spl_m(G)) = E(G) \cup \{ v_i^k v_j : v_i v_j \in E(G),\, 1 \le k \le m \}.
	\]
\end{definition}
\begin{definition}
	Let $G$ be a graph with vertex set $V(G)=\{v_1,\dots,v_n\}$.  
	The shadow graph of $G$, denoted by $D_2(G)$, is the graph obtained by taking two copies of $G$, say $G^1$ and $G^2$, and adding edges between the copies as follows: for every edge $v_i v_j\in E(G)$, we join the vertex $v_i^1$ in $G^1$ to $v_j^2$ in $G^2$, and the vertex $v_i^2$ in $G^2$ to $v_j^1$ in $G^1$. No new edges are added within $G^1$ or within $G^2$.
\end{definition}

\begin{definition}
	For a positive integer $m$, the {$m$-shadow graph} of $G$, denoted by $D_m(G)$, is defined by taking $m$ copies $G^1,G^2,\dots,G^m$ of $G$ with vertex sets
	\[
	V(G^r)=\{v_{r,1},v_{r,2},\dots,v_{r,n}\}, \qquad r=1,\dots,m,
	\]
	and adding edges between every pair of distinct copies as follows:
	whenever $v_i v_j \in E(G)$, we join $v_{r,i}$ to $v_{s,j}$ for all $r\ne s$.
	No additional edges are added inside any individual copy $G^r$.
\end{definition}

%	\begin{theorem}\cite{9}
	%		The spectra of $m-$splitting graph is 
	%		
	%		\small
	%		$\begin{Bmatrix}
		%			0\lambda_1 & \ldots & 0\lambda_n & \frac{1+\sqrt{1+4m}}{2}\lambda_1& \ldots & \frac{1+\sqrt{1+4m}}{2}\lambda_n& \frac{1-\sqrt{1+4m}}{2}\lambda_1& \ldots&\frac{1-\sqrt{1+4m}}{2}\lambda_n\\
		%			m-1& \ldots & m-1& 1 & \ldots& 1& 1& \ldots& 1
		%		\end{Bmatrix}$
	%	\end{theorem}
%		\begin{theorem}\cite{9}
	%		The spectra of $m-$shadow graph is 
	%		
	%		\small
	%		$\begin{Bmatrix}
		%			0\lambda_1 & \ldots & 0\lambda_n & m\lambda_1& \ldots & m\lambda_n\\
		%			m-1& \ldots & m-1& 1 & \ldots& 1
		%		\end{Bmatrix}$
	%	\end{theorem}

\section{Vertex energy of $m-$splitting graph}
\begin{theorem}\label{thm1}
	Let $G$ be a graph with adjacency matrix $A$. Let $\mathrm{Spl}_m(G)$ be the $m$-splitting graph of $G$. Then the vertex energies satisfy
	\[
	\mathcal{E}_{\mathrm{Spl}_m(G)}(v_i)=
	\begin{cases}
		\dfrac{2m+1}{\sqrt{4m+1}}\,\mathcal{E}_G(v_i), & v_i\in V(G),\\[8pt]
		\dfrac{2}{\sqrt{4m+1}}\,\mathcal{E}_G(v_i), & v_i\in \mathrm{Spl}_m(G)\setminus V(G).
	\end{cases}
	\]
\end{theorem}

\begin{proof}
	Let $A x=\lambda x$ be an eigenpair of $A$ with $\|x\|=1$. Arrange vertices of $\mathrm{Spl}_m(G)$ so that the first $n$ coordinates correspond to $V(G)$ and the remaining $mn$ coordinates correspond to the $m$ copies. The adjacency matrix of $\mathrm{Spl}_m(G)$ has the block form
	$$B=\begin{pmatrix}
		A & A\otimes \mathbf{1}_{1\times m}\\[4pt]
		A\otimes \mathbf{1}_{m\times 1} & 0_{mn}
	\end{pmatrix},$$
	and we find eigenvectors of $B$ of the form
	$$y=\begin{pmatrix}
		x \\ c(1_m\otimes x)
	\end{pmatrix},$$
	where $1_m\otimes x$ denotes the stacked vector consisting of $m$ copies of $x$ and $c\in\mathbb{R}$ (or $\mathbb{C}$).
	
	Compute
	$$B y
	=
	\begin{pmatrix}
		(1+m c)\lambda x \\[4pt]
		\lambda(1_m\otimes x)
	\end{pmatrix}.$$
	Thus $y$ is an eigenvector with eigenvalue $\mu$ precisely when
	$$\mu x=(1+m c)\lambda x,\qquad \mu c(1_m\otimes x)=\lambda(1_m\otimes x).$$
	From the second relation $\mu c=\lambda$, i.e. $c=\dfrac{\lambda}{\mu}$. Substituting into the first gives the quadratic
	$\mu^2-\lambda\mu-m\lambda^2=0.$
	
	Hence the two corresponding eigenvalues are
	$$\mu^\pm=\lambda\alpha_\pm,\qquad \alpha_\pm=\frac{1\pm\sqrt{1+4m}}{2}.$$
	Note that $\sqrt{1+4m}>1$, so $\alpha_+>0$, $\alpha_-<0$, and $\alpha_+\alpha_-=-m$.
	
	For each sign $\pm$ the corresponding (unnormalized) eigenvector uses $c=\dfrac{1}{\alpha_\pm}$. The normalization factor is
	$$\|y^\pm\|^2=1+m\frac{1}{\alpha_\pm^2}=\frac{\alpha_\pm^2+m}{\alpha_\pm^2}.$$
	Therefore a unit eigenvector can be written as
	$$y^\pm=\begin{pmatrix}
		\dfrac{|\alpha_\pm|}{\sqrt{\alpha_\pm^2+m}}\, x \\[6pt]
		\dfrac{\operatorname{sgn}(\alpha_\pm)}{\sqrt{\alpha_\pm^2+m}}\,(1_m\otimes x)
	\end{pmatrix}.$$
	\begin{enumerate}
		\item 	For an original vertex $v_k\in V(G)$, the $k$-th coordinate in $y^\pm$ equals $$\dfrac{|\alpha_\pm|}{\sqrt{\alpha_\pm^2+m}}\,x_k.$$ Thus the contribution of $(\mu^\pm,y^\pm)$ is
		$$|\lambda|\,|x_k|^2\, \frac{|\alpha_\pm|\alpha_\pm^2}{\alpha_\pm^2+m}.$$
		Summing over both signs gives
		$$S_1=\sum_{\pm}\frac{|\alpha_\pm|\alpha_\pm^2}{\alpha_\pm^2+m}.$$
		Similarly,
		$$S_2=\sum_{\pm}\frac{|\alpha_\pm|}{\alpha_\pm^2+m}
		=\frac{1+\sqrt{1+4m}}{1+\sqrt{1+4m}+4m}
		-\frac{1-\sqrt{1+4m}}{1-\sqrt{1+4m}+4m}
		=\frac{2}{\sqrt{1+4m}}.$$
		Also,
		$\alpha_+-\alpha_-=\sqrt{1+4m}.$
		Hence
		$$S_1=(\alpha_+-\alpha_-)-mS_2
		=\sqrt{1+4m}-m\cdot\frac{2}{\sqrt{1+4m}}
		=\frac{(1+4m)-2m}{\sqrt{1+4m}}
		=\frac{2m+1}{\sqrt{1+4m}}.$$
		
		Therefore,
		$$\mathcal{E}_{\mathrm{Spl}_m(G)}(v_k)
		=\sum_{\lambda,x}|\lambda|\,|x_k|^2\,S_1
		=\frac{2m+1}{\sqrt{4m+1}}\,\mathcal{E}_G(v_k).$$
		\item  	For a split vertex $v_{n+k}$, the relevant coordinate is $$\dfrac{\operatorname{sgn}(\alpha_\pm)}{\sqrt{\alpha_\pm^2+m}}x_k.$$ Hence the contribution is
		$$|\lambda|\,|x_k|^2\,\frac{|\alpha_\pm|}{\alpha_\pm^2+m},$$
		and summing over both signs yields $$S_2=\dfrac{2}{\sqrt{1+4m}}.$$ Therefore,
		$$\mathcal{E}_{\mathrm{Spl}_m(G)}(v_{n+k})
		=\frac{2}{\sqrt{4m+1}}\,\mathcal{E}_G(v_k).$$
		
	\end{enumerate}
	
\end{proof}

\begin{remark}
	Summing over all original vertices and their $m$ copies gives
	
	\begin{align*}
		\mathcal{E}(\mathrm{Spl}_m(G))
		&=\sum_{k=1}^n\Big(\mathcal{E}_{\mathrm{Spl}_m(G)}(v_k)+\sum_{j=1}^m\mathcal{E}_{\mathrm{Spl}_m(G)}(v_k^j)\Big)
		\\&=\Big(\frac{2m+1}{\sqrt{4m+1}}+m\cdot\frac{2}{\sqrt{4m+1}}\Big)\mathcal{E}(G)\\&
		=\sqrt{4m+1}\,\mathcal{E}(G).
	\end{align*}
	Thus, the total energy of the $m-$splitting graph is exactly $\sqrt{4m+1}$ times the energy of the original graph.
\end{remark}
\begin{corollary}
	Let $G$ be a graph with adjacency matrix $A$. Then the vertex energy of the splitting graph $S'(G)$ satisfies \begin{align*}
		\mathcal{E}_{S'(G)}(v_i)=
		\begin{cases}
			\dfrac{3}{\sqrt{5}}\,\mathcal{E}_G(v_i), & v_i\in V(G),\\[8pt]
			\dfrac{2}{\sqrt{5}}\,\mathcal{E}_G(v_i), & v_i\in S'(G)\setminus V(G).
		\end{cases}
	\end{align*}
\end{corollary}

\begin{proof}
	Proof is direct by substituting $m=1$ in Theorem \ref{thm1}.
	%	Let $A x=\lambda x$ be an eigenpair of $A$ with $\|x\|=1$. 
	%	The adjacency matrix of $S'(G)$ is
	%	$$
	%	B=\begin{pmatrix}A & A\\ A & 0\end{pmatrix}.
	%	$$
	%	Seeking an eigenvector of $B$ of the form $(x,cx)^T$ gives the quadratic
	%	$$
	%	\mu^2-\lambda\mu-\lambda^2=0,
	%	$$
	%	so the two eigenvalues are $\mu^\pm=\lambda\alpha_\pm$, where $\alpha_\pm=\tfrac{1\pm\sqrt{5}}{2}$.
	%	
	%	A corresponding eigenvector is
	%	$$
	%	y^\pm=\begin{pmatrix}\dfrac{\alpha_\pm}{\sqrt{1+\alpha_\pm^2}}x \\[6pt] \dfrac{1}{\sqrt{1+\alpha_\pm^2}}x\end{pmatrix}.
	%	$$
	%	For an original vertex $v_k\in V(G)$, the contribution of $(\mu^\pm,y^\pm)$ is
	%	$$
	%	|\mu^\pm|\,|(y^\pm)_k|^2=|\lambda|\,|x_k|^2\,\frac{|\alpha_\pm|\alpha_\pm^2}{1+\alpha_\pm^2}.
	%	$$
	%	Summing over $\pm$ gives the factor
	%	$$
	%	S_1=\sum_{\pm}\frac{|\alpha_\pm|\alpha_\pm^2}{1+\alpha_\pm^2}=\frac{3}{\sqrt{5}}.
	%	$$
	%	Hence
	%	$$
	%	\mathcal{E}_{S'(G)}(v_k)=\frac{3}{\sqrt{5}}\,\mathcal{E}_G(v_k).
	%	$$
	%	
	%	For a split vertex $v_{n+k}$, the contribution of $(\mu^\pm,y^\pm)$ is
	%	$$
	%	|\lambda|\,|x_k|^2\,\frac{|\alpha_\pm|}{1+\alpha_\pm^2},
	%	$$
	%	and the sum over $\pm$ gives
	%	$$
	%	S_2=\sum_{\pm}\frac{|\alpha_\pm|}{1+\alpha_\pm^2}=\frac{2}{\sqrt{5}}.
	%	$$
	%	Thus
	%	$$
	%	\mathcal{E}_{S'(G)}(v_{n+k})=\frac{2}{\sqrt{5}}\,\mathcal{E}_G(v_k).
	%	$$
	%	
	%	Finally, summing over all vertices,
	%	$$
	%	\mathcal{E}(S'(G))=\sum_{k=1}^n\big(\mathcal{E}_{S'(G)}(v_k)+\mathcal{E}_{S'(G)}(v_{n+k})\big)
	%	=\sqrt{5}\,\mathcal{E}(G).
	%	$$
	%	This completes the proof.
\end{proof} 
\begin{remark}
	Summing over all vertices $k=1,\dots,n$ gives
	$$\mathcal{E}(S'(G))
	=\sum_{k=1}^n\Big(\mathcal{E}_{S'(G)}(v_k)+\mathcal{E}_{S'(G)}(v_{n+k})\Big)
	=\frac{3+2}{\sqrt{5}}\,\sum_{k=1}^n \mathcal{E}_G(v_k)
	=\sqrt{5}\,\mathcal{E}(G).
	$$
	Thus, the total energy of the splitting graph is exactly $\sqrt{5}$ times the energy of the original graph.
\end{remark}

\section{Vertex energy of the $m-$shadow graph}
\begin{theorem}\label{thm:shadow}
	Let $G$ be a graph with adjacency matrix $A$. Let $D_m(G)$ be the $m$-shadow graph obtained by taking $m$ copies of $G$ and joining every vertex in each copy to all neighbours of the corresponding vertex in every other copy. Then, for every copy index $r\in\{1,\dots,m\}$ and every vertex $v_k\in V(G)$,
	\[
	\mathcal{E}_{D_m(G)}(v_{r,k})=\mathcal{E}_G(v_k).
	\]
\end{theorem}

\begin{proof}
	Let $A x=\lambda x$ be an eigenpair of $A$ with $\|x\|=1$. The adjacency matrix of $D_m(G)$ is $B=J_m\otimes A$, where $J_m$ is the $m\times m$ all-ones matrix. It is known that $J_m$ has one eigenvalue $m$ with eigenvector $1_m$, and $m-1$ eigenvalues equal to $0$ with eigenvectors orthogonal to $1_m$.
	
	Using the Kronecker product property, from the eigenpair $(\lambda,x)$ of $A$ we obtain:
	\[
	B(1_m\otimes x) = (J_m 1_m)\otimes (Ax) = m\lambda (1_m\otimes x),
	\]
	so $(m\lambda,\,1_m\otimes x)$ is an eigenpair of $B$, and $m\lambda$ is the only nonzero eigenvalue generated from $\lambda$. The remaining $m-1$ eigenvectors are of the form $z\otimes x$ with $z\perp 1_m$, and they satisfy $B(z\otimes x)=0$.
	
	Since $\|1_m\|^2=m$ and $\|x\|=1$, we have $\|1_m\otimes x\|^2=m$. Thus
	\[
	y=\frac{1}{\sqrt{m}}\,(1_m\otimes x)
	\]
	is a unit eigenvector corresponding to the eigenvalue $m\lambda$. Under the vertex ordering
	
	\noindent $v_{1,1},\dots,v_{n,1},v_{1,2},\dots,v_{n,m}$, the entry of $y$ at the vertex $v_{r,k}$ equals $\dfrac{1}{\sqrt{m}}\,x_k$.
	
	Hence the contribution of the eigenpair $(m\lambda,y)$ to the vertex-energy of $v_{r,k}$ is
	\[
	|m\lambda|\left|\frac{1}{\sqrt{m}}x_k\right|^2
	=|\lambda|\,|x_k|^2.
	\]
	
	If $\mu=0$ is any eigenvalue coming from $z\otimes x$, then its contribution to vertex-energy is
	\[
	|\mu|\,|(\text{eigenvector coordinate})|^2 = 0,
	\]
	so all zero eigenvalues give no contribution and are omitted from the summation.
	
	Therefore, summing over all eigenpairs $(\lambda,x)$ of $A$, we obtain
	\[
	\mathcal{E}_{D_m(G)}(v_{r,k})
	=\sum_{\lambda,x}|\lambda|\,|x_k|^2
	=\mathcal{E}_G(v_k).
	\]
	
\end{proof}

\begin{remark}
	Summing over $k$ and over the $m$ copies yields the total energy
	$D_m(G)$,
	\[
	E(D_m(G))
	=\sum_{r=1}^m \sum_{k=1}^n \mathcal{E}_{D_m(G)}(v_{r,k})
	=m \sum_{k=1}^n \mathcal{E}_G(v_k)
	=m\,E(G).
	\]
	Hence the total energy of the $m-$shadow graph is exactly $m$ times the energy of the original graph.	
\end{remark}
\begin{corollary}
	Let $G$ be a graph with adjacency matrix $A$. Let $D_2(G)$ denote the shadow graph of $G$.  Then
	\begin{align*}
		\mathcal{E}_{D_2(G)}(v_{r,i}) &= \mathcal{E}_G(v_i)\qquad\text{for all }r\in\{1,2\},\; i\in\{1,\dots,n\},
	\end{align*}
\end{corollary}

\begin{proof}
	This follows immediately from Theorem \ref{thm:shadow} by taking $m=2$.
\end{proof}

\begin{remark}
	
	Summing the equalities of the corollary over all $k=1,\dots,n$ gives
	\[
	\mathcal{E}(D_2(G))=\sum_{k=1}^n\big(\mathcal{E}_{D_2(G)}(v_k)+\mathcal{E}_{D_2(G)}(v_k')\big)
	=2\sum_{k=1}^n\mathcal{E}_G(v_k)
	=2\,\mathcal{E}(G).
	\]
	Hence the total energy of the shadow graph is twice the energy of $G$.
\end{remark}	\section{Conclusion}
In this work, we investigated the distribution of graph energy among individual vertices for two derived graph operations: the $m$-splitting graph and the $m$-shadow graph. For the $m$-splitting graph $\mathrm{Spl}_m(G)$, we established explicit scaling relations that show how vertex energies of the original graph are redistributed among both original and newly added vertices. In particular, it was proven that the total energy of $\mathrm{Spl}_m(G)$ increases by a factor of $\sqrt{4m+1}$ relative to the base graph.

For the $m$-shadow graph $D_m(G)$, we proved that each original vertex and each corresponding shadow vertex have identical vertex energies.

These results provide a clear understanding of how structural graph transformations influence spectral characteristics at the level of individual vertices.

%	\section{Declaration}\label{sec14}
%	\begin{itemize}
	%		\item Funding: Not applicable.
	%		\item Conflict of interest/Competing interests: There are no conflict of interest.
	%		\item Ethics approval: Not applicable.
	%		\item Consent to participate: Not applicable.
	%		\item Consent for publication: Not applicable.
	%		\item Availability of data and materials: Not applicable.
	%		\item Code availability: Not applicable.
	%		\item Authors' contributions: All authors contributed equally to this work.
	%	\end{itemize}

\end{document}